\documentclass[12pt]{amsart}
\usepackage[T1]{fontenc}
\usepackage[latin1]{inputenc}
\usepackage{typearea}
\usepackage{geometry}
\usepackage{ulem}
\usepackage{amsmath}
\usepackage{amssymb}
\usepackage{latexsym}
\usepackage{enumerate}
\usepackage{amsthm}
\usepackage[all]{xy}
\usepackage{hhline}
\usepackage{epsf} 
\usepackage{cite}
\newtheorem{theorem}{{\sc Theorem}}[section]
\newtheorem{lemma}[theorem]{{\sc Lemma}}
\newtheorem{proposition}[theorem]{{\sc Proposition}}

\theoremstyle{remark}
\newtheorem{remark}[theorem]{{\sc Remark}}

\theoremstyle{definition}

\newcommand{\R}{\mathbb{R} }
\newcommand{\N}{\mathbb{N} }

\newcommand{\B}{\mathcal{B}}
\newcommand{\F}{\mathcal{F}}

\newcommand{\D}{\mathcal{D}}
\newcommand{\W}{\mathcal{W}}
\newcommand{\K}{\mathcal{K}}
\newcommand{\Z}{\mathbb{Z}}

\newcommand{\calL}{\mathcal{L}}

\newcommand{\htilde}{\tilde{h}}
\newcommand{\ftilde}{\tilde{f}}

\providecommand{\abs}[1]{\lvert #1\rvert}

\providecommand{\fnorm}[1]{\lVert #1\rVert_\infty}

\renewcommand{\phi}{\varphi}
\renewcommand{\epsilon}{\varepsilon}

\renewcommand{\rho}{\varrho}

\begin{document}
\title[Stein's method for the half-normal distribution with applications]{Stein's method for the half-normal distribution with applications to limit theorems related to the simple symmetric random walk}
\author{Christian D\"obler}
\thanks{Universit\'{e} du Luxembourg, Unit\'{e} de Recherche en Math\'{e}matiques \\
christian.doebler@uni.lu }
\begin{abstract}
We develop Stein's method for the half-normal distribution and apply it to derive rates of convergence in distributional limit theorems for three statistics of the simple symmetric random walk: the maximum value, the number of returns to the origin and the number of sign changes up to a given time $n$. 
We obtain explicit error bounds with the optimal rate $n^{-1/2}$ for both the Kolmogorov and the Wasserstein metric. 
In order to apply Stein's method, we compare the characterizing operator of the limiting half-normal distribution with suitable characterizations of the discrete approximating distributions, exploiting a recent technique by Goldstein and Reinert \cite{GolRei13}.  
\end{abstract}

\maketitle

\section{Introduction}\label{intro}
This article concerns the rate of convergence issue for three limit theorems in the surroundings of the one-dimensional simple symmetric random walk (SRW). By this we mean 
the discrete time stochastic process $(S_n)_{n\geq0}$ defined by $S_0:=0$ and $S_n:=\sum_{j=1}^n X_j$, $n\geq1$, where $X_1,X_2,\dotsc,$ are iid random variables 
with $P(X_1=1)=P(X_1=-1)=1/2$. It has been known for a long time that various statistics of the process $(S_n)_{n\geq0}$ exhibit a quite  counter intuitive distributional 
behaviour, see e.g. Chapter 3 of \cite{Fel1} for some qualitative limit theorems.\\
As to quantitative results, in \cite{Doe12b} a rate of convergence for the arcsine law was proved using Stein's method for Beta distributions (see also \cite{GolRei13} for an improvement of this result with respect to the computation of an explicit constant of convergence and also \cite{Doe14} for a further improvement of this constant).\\ 
Here, we will focus on limit theorems which state convergence towards the distribution $\mu$ of $Y:=\abs{Z}$, where $Z\sim N(0,1)$ is standard normally distributed. This distribution is commonly known as the (standard) \textit{half-normal distribution}. The technique of proof will be to compare the Stein characterization of $Y$ with a suitable characterization for the approximating discrete distribution, as was proposed by Goldstein and Reinert \cite{GolRei13} and also used in \cite{Doe12b}. Generally, this technique is promising, whenever a concrete formula for the probability mass function of the discrete distribution is at hand, which yields a Stein characterization similar to the one for the limiting distribution.
 Note that even though the probability mass function and, hence, the discrete distribution must be known for this approach to be applicable, it might still be of interest to approximate by an easier to handle absolutely continuous distribution. In fact, already the de Moivre-Laplace theorem gives an approximation to the known binomial by the normal distribution.\\
In the case of $Y=\abs{Z}$ a suitable Stein characterization is easily found, using the \textit{density approach} of Stein's method (see \cite{ChSh}, \cite{EiLo10} or \cite{CGS} for instance). Although we could simply quote the theory and general bounds on the solution to the Stein equation from \cite{CGS} or \cite{ChSh}, 
for 
example, we prefer deriving our own bounds, which usually yield better constants.\\\\
 The rate of convergence results in this paper are always with respect to a certain probability metric, which is defined via test functions. Thus, if $\mu$ and $\nu$ are two probability measures on $(\R,\B)$ and $\mathcal{H}$ is some class of measurable test functions that are integrable with respect to $\mu$ and $\nu$, then we define 
the distance
\begin{equation*}
 d_\mathcal{H}(\mu,\nu):=\sup_{h\in\mathcal{H}}\Bigl|\int_\R h d\mu-\int_\R h d\nu\Bigr|\,.
\end{equation*}
For example, if $\mathcal{W}$ is the class of Lipschitz continuous functions on $\R$ with Lipschitz constant not greater than $1$ and if $\mu$ and $\nu$ both have first moments, then $d_\W(\mu,\nu)$ is the \textit{Wasserstein distance} between $\mu$ and $\nu$. On the other hand, if $\K$ is the class of functions $h_z:=1_{(-\infty,z]}$, $z\in\R$, then 
we obtain the \textit{Kolmogorov distance}
\[d_\K(\mu,\nu)=\sup_{z\in\R}\Bigl|\mu\bigl((-\infty,z]\bigr)-\nu\bigl((-\infty,z]\bigr)\Bigr| \,,\]
which is particularly natural from a statistician' s point of view.
For real-valued random variables $X$ and $Y$ we write $d_\mathcal{H}(X,Y)$ for $d_\mathcal{H}\bigl(\calL(X),\calL(Y)\bigr)$.\\\\
Now, we introduce the statistics of $(S_n)_{n\geq0}$, which converge in distribution to $Y$. First, consider the number $K_n$ of times that the random walk returns to the origin up to time $n$, i.e.
\begin{equation}\label{Kn}
 K_n:=\bigl|\{1\leq k\leq n=2m\,:\,S_k=0\}\bigr|\,.
\end{equation}
The variable $K_n$ is sometimes called the \textit{local time} or \textit{occupation time} at $0$ by time $n$.
From Theorem 7, Section 5 of \cite{Fel49} (see also Equation (5.31) there) it is known that $K_n/\sqrt{n}\stackrel{\D}{\rightarrow} Y$ as $n\to\infty$ (see also Equation (1) of \cite{CH49}). 
Note that intuition might lead us to the (false) conclusion that the number of returns to the origin should roughly grow linearly with the time $n$. Here is a theorem which gives error bounds for this distributional convergence.

\begin{theorem}\label{theo2}
Let $n=2m$ be an even positive integer. Then, with $W:=W_n:=K_n/\sqrt{n}$ 
\begin{align*}
d_\W(W,Y)&\leq\frac{1}{\sqrt{n}}\left(\frac{2}{\pi}+2\right)+\frac{1}{n}\sqrt{\frac{2}{\pi}}\quad\text{and}\\
d_\K(W,Y)&\leq\frac{1}{\sqrt{n}}\left(\frac{3+2\sqrt{2}}{\sqrt{2\pi}}+\frac{3}{4}\right)+\frac{3}{2n}\,.
\end{align*}
\end{theorem}

It should be mentioned that the rate $n^{-1/2}$ for $d_\K(W,Y)$ was also given in \cite{PRR13} but they did not compute an explicit constant.
Next, consider 
\begin{equation}\label{Mn}
 M_n:=\max_{0\leq k\leq n} S_k \in\{0,1,\dots,n\}\,.
\end{equation}
Then, $M_n/\sqrt{n}\stackrel{\D}{\rightarrow} Y$ as $n\to\infty$. This follows by an application of the CLT to the result of Theorem 1 in Section 7 of Chapter 3 in \cite{Fel1}. We will prove the following quantitative version of this result.

\begin{theorem}\label{theo1}
Let $n=2m$ be an even positive integer. Then, with $W:=W_n:=M_n/\sqrt{n}$ 
\begin{align*}
d_\W(W,Y)&\leq\frac{1}{\sqrt{n}}\Bigl(3+\frac{2}{\pi}\Bigr)\quad\text{and}  \\
d_\K(W,Y)&\leq\frac{1}{\sqrt{n}}\Bigl(4\sqrt{\frac{2}{\pi}}+\frac{1}{2}\Bigr)+\frac{2}{n}\,.
\end{align*}
\end{theorem}

Finally, consider the number $C_{n}$ of sign changes by the random walk up to time $n:=2m+1$, $m\in\N$, i.e.
\begin{equation}\label{Cn}
 C_n:=C_{2m+1}:=\bigl|\{1\leq k\leq 2m\,:\, S_{k-1}\cdot S_{k+1}<0\}\bigr|\,.
\end{equation}
Then, Theorem 2 in Section 5 of Chapter 3 in \cite{Fel1} states that $2C_{2m+1}/\sqrt{2m+1}\stackrel{\D}{\rightarrow} Y$ as $m\to\infty$. Again, we obtain a quantitative version of this result. 

\begin{theorem}\label{theo3}
Let $m$ be a positive integer and $n:=2m+1$.\\
 Then, with $W:=W_m:=2C_{2m+1}/\sqrt{2m+1}$ 
\begin{align*}
d_\W(W,Y)&\leq\frac{1}{\sqrt{n}}\Bigl(4+\frac{2}{\pi}\Bigr)+\sqrt{\frac{2}{\pi}}\frac{1}{n}
+\frac{2\sqrt{2}}{\pi}\frac{1}{n^{3/2}}\quad\text{and}\\
d_\K(W,Y)&\leq\frac{1}{\sqrt{n}}\Bigl(\frac{2\sqrt{2}+4}{\sqrt{\pi}}+\frac{3}{2}\Bigr)+\frac{3}{n}+\frac{4}{\sqrt{\pi}}\frac{1}{n^{3/2}}\,.
\end{align*}
\end{theorem}
 
The rest of the paper is organized as follows. In Section \ref{Stein} we develop Stein's method for the half-normal distribution of $Y$ and review the technique by Goldstein and Reinert \cite{GolRei13} of comparison with a discrete distribution. In Sections \ref{returns}, \ref{max} and \ref{signs} we present the proofs of Theorems \ref{theo2}, 
\ref{theo1} and \ref{theo3}, respectively and in Section \ref{rates} we show the optimality of the obtained convergence rates. Finally, in Section \ref{proofs} we give proofs of some of the results from Section \ref{Stein}. 

\section{Stein's method for the half-normal distribution and for discrete distributions}\label{Stein}
Stein's method is by now a well-established device for proving concrete error bounds in distributional convergence problems. 
Since its introduction by Stein in the seminal paper \cite{St72} on univariate normal approximation for sums of random variables, satisfying a certain mixing condition, it has undergone remarkable progress. On the one hand, the range of normal approximation problems that can be tackled by means of the method has been largely extended, particularly due to the development of certain coupling constructions (see \cite{CGS} for an introduction and overview). On the other hand, the essential idea of characterizing a given distribution by a certain differential or difference equation, was succesfully carried over to other prominent distributions, like, for instance, the Poisson distribution (see e.g. \cite{Ch75} and \cite{BHJ}), the Gamma distribution (see \cite{Luk}), the exponential distribution (see e.g. \cite{CFR11}, \cite{PekRol11} or \cite{FulRos12}), the Beta distribution (see \cite{GolRei13}, \cite{Doe12c} and \cite{Doe14}), the Laplace distribution (see \cite{PiRen12}) and, more generally, the recent article \cite{Gau14} on the class of Variance-Gamma distributions. Furthermore, general techniques have been proposed to develop Stein's method for 
a distribution with a given density, like for example the density approach (see \cite{DHRS}, \cite{ChSh}, \cite{EiLo10} or \cite{CGS}) or the general approach in \cite{Doe14}, which is adapted to a given exchangeable pair. For a nice recent generalization of the density approach also see \cite{LRS14}. \\
Before we develop Stein' s method for the half-normal distribution, let us make the following remark.
Note that since $Y=\abs{Z}$ and the random variable $W$ in each of the Theorems \ref{theo1}-\ref{theo3} is nonnegative, it would in principle be possible to apply Stein's method for the standard normal distribution for their 
proofs. Indeed, if $\mathcal{H}$ is a given class of test functions on $[0,\infty)$ and if for $h\in\mathcal{H}$ we define the function $g$ on $\R$ by $g(x):=h(\abs{x})$ and denote by $\mathcal{G}$ the class of all those functions $g$, when $h$ is running through $\mathcal{H}$, we have that 
\begin{align*}
d_\mathcal{H}(W,Y)&=\sup_{h\in\mathcal{H}}\bigl|E[h(W)]-E[h(Y)]\bigr|
=\sup_{g\in\mathcal{G}}\bigl|E[g(W)]-E[g(Z)]\bigr|\\
&=\sup_{g\in\mathcal{G}}\bigl|E\bigl[\ftilde_g'(W)-W\ftilde_g(W)\bigr]\bigr|\,,
\end{align*}
where we denote by $\ftilde_g$ the standard solution to the standard normal Stein equation corresponding to the test function $g$. Unfortunately, unless Lipschitz-continuous test functions are considered, in general the bounds on the functions $\ftilde_g$, $g\in\mathcal{G}$, and on their lower order derivatives, which could be derived from known results from Stein's method of normal approximation, are worse than the bounds we obtain for the functions $f_h$ in Lemmas \ref{bounds} and \ref{bounds2}.
As a consequence, following this route would yield worse constants in Theorems \ref{theo2}-\ref{theo3} than we will obtain by 
considering Stein' s method for the half-normal distribution itself, at least as far as the Kolmogorov distance is concerned.\\
 However, for  some of our bounds, we actually will use results from Stein's method for the standard normal distribution observing that for $x\geq0$ the solution $f_h(x)$ to the half-normal Stein equation 
\eqref{steineq} given by \eqref{fh1}, \eqref{fh2} and $\ftilde_g(x)$ coincide, which makes it possible to shorten the exposition of some of our proofs. In our opinion, this compromise eventually justifies developing a version of Stein' s method for the half-normal distribution instead of just applying existing results on normal approximation.\\\\
We begin by developing Stein's method for the half-normal distribution. Note that $\mu$ is supported on $[0,\infty)$, since $Y=\abs{Z}$, where $Z\sim N(0,1)$. 
We denote by $p$ and $F$ the (continuous) density function and distribution function of $Y$, respectively. Thus, as a trivial computation shows, we have 
\begin{align}
p(x)&=2\phi(x)1_{(0,\infty)}(x)=\sqrt{\frac{2}{\pi}}e^{-x^2/2}1_{(0,\infty)}(x)\text{  and}\label{pdf}  \\
F(x)&=(2\Phi(x)-1)1_{(0,\infty)}(x)\label{df}\,,
\end{align}
where $\phi,\Phi$ denote the (continuous) density function and distribution function of $Z$, respectively.
According to the density approach in Stein's method, we have the following result.

\begin{proposition}[Stein characterization]\label{sc}
 A random variable $X$ with values in $[0,\infty)$ has the half-normal distribution $\mu$ if and only if 
\begin{equation*}
 E\bigl[f'(X)\bigr]=E\bigl[Xf(X)\bigr]-f(0)\sqrt{\frac{2}{\pi}}
\end{equation*}
for all functions $f:[0,\infty)\rightarrow\R$, which are absolutely continuous on every compact sub-interval of $[0,\infty)$ such that $E\abs{f'(Y)}<\infty$.
\end{proposition}
As we do not explicitly need the result of Proposition \ref{sc}, we omit the rather standard proof.
For a given measurable function $h$ on $[0,\infty)$ with $E\abs{h(Y)}<\infty$ Proposition \ref{sc} now motivates the following \textit{half-normal Stein equation}
\begin{equation}\label{steineq}
 f'(x)-xf(x)=h(x)-\mu(h)\,,
\end{equation}
where we abbreviate $\mu(h):=E[h(Y)]$. This equation is to be solved for the function $f$ on $[0,\infty)$. If $f$ is a solution to \eqref{steineq} and $W$ is a given 
nonnegative random variable, then, taking expectations, we have the following identity:
\begin{equation}\label{id1}
 E[h(W)]-E[h(Y)]=E\bigl[f'(W)-Wf(W)]
\end{equation}
As a matter of fact, the right hand side of \eqref{id1} may often be bounded more easily (even uniformly in the test functions $h$ from some class $\mathcal{H}$ of functions) 
than the left hand side, if one further tool is available. This additional tool may be a coupling, like for example the exchangeable pairs coupling (see \cite{CGS}), 
or a characterization for the distribution $\calL(W)$, as will be exploited in this paper.\\
We now introduce the standard solution $f_h$ to \eqref{steineq}. Let $f_h:[0,\infty)\rightarrow\R$ be defined by
\begin{align}
f_h(x)&:=\frac{1}{p(x)}\int_0^x\bigl(h(t)-\mu(h)\bigr)p(t)dt\notag\\
&=\frac{1}{\phi(x)} \int_0^x\bigl(h(t)-\mu(h)\bigr)\phi(t)dt\label{fh1}\\
&=-\frac{1}{\phi(x)} \int_x^\infty\bigl(h(t)-\mu(h)\bigr)\phi(t)dt\label{fh2}\,.
\end{align}

It is easily checked that $f_h$ indeed solves Equation \eqref{steineq} and that the solutions of the homogeneous equation corresponding to \eqref{steineq} 
have the form $ce^{x^2/2}$ for some constant $c\in\R$. This particularly shows that if $f_h$ is bounded, then it is the only bounded solution to  \eqref{steineq} and 
also the only solution $f$ with $\lim_{x\to\infty}e^{-x^2/2}f(x)=0$.
Note that the half-normal Stein equation \eqref{steineq} is the same as the standard normal one (see e.g. \cite{CGS}) except that we only consider functions on $[0,\infty)$, here. Thus, there should also be some correspondence between the solutions. Indeed, if as above, for given $h$ on $[0,\infty)$, we consider the function $g$ on $\R$ given by 
$g(x):=h(\abs{x})$ and denote by $\ftilde_g$ the standard solution to the Stein equation for the standard normal distribution $\calL(Z)$ and the test function $g$, we obtain for each $x\geq0$ that
\begin{align}\label{coreq1}
\ftilde_g(x)&=-\frac{1}{\phi(x)}\int_x^\infty\bigl(g(t)-E[g(Z)]\bigr)\phi(t)dt\notag\\
&=-\frac{1}{\phi(x)}\int_x^\infty\bigl(h(t)-E[h(Y)]\bigr)\phi(t)dt\notag\\
&=f_h(x)
\end{align}
by \eqref{fh2}. Thus, $\ftilde_g$ coincides with $f_h$ on $[0,\infty)$. This allows us to derive properties of the solutions $f_h$ to \eqref{steineq} from those of the functions $\ftilde_g$, which are well-studied. 
For example note that, if $h$ is Lipschitz on $[0,\infty)$ with Lipschitz constant $L>0$, then for $x,y\in\R$ 
\begin{equation}\label{coreq2}
 \bigl|g(x)-g(y)\bigr|=\bigl|h(\abs{x})-h(\abs{y})\bigr|\leq L\bigl|\abs{x}-\abs{y}\bigr|\leq L\abs{x-y}\,.
\end{equation}
So, $g$ is also Lipschitz with the same constant $L$. \\
In order to make good use of identity \eqref{id1} one needs bounds on the solutions $f_h$ and 
their lower order derivatives. The following lemma gives bounds for bounded measurable or Lipschitz test functions $h$. 
\begin{lemma}\label{bounds}
Let $h:[0,\infty)\rightarrow\R$ be Borel-measurable.
\begin{enumerate}[{\normalfont(i)}]
 \item If $h$ is bounded, then $f_h$ is Lipschitz and with $z_{0.75}:=\Phi^{-1}(3/4)$ we have
\begin{enumerate}[{\normalfont (a)}]
\item $\fnorm{f_h}\leq\frac{\fnorm{h-\mu(h)}}{4\phi(z_{0.75})}\,,$
\item $\fnorm{f_h'}\leq 2\fnorm{h-\mu(h)}\,.$
\end{enumerate}
\item If $h$ is Lipschitz, then $f_h$ is continuously differentiable with a Lipschitz continuous derivative and we have the bounds
\begin{enumerate}[{\normalfont (a)}]
\item $\fnorm{f_h}\leq\fnorm{h'}\,,$
\item $\fnorm{f_h'}\leq\sqrt{\frac{2}{\pi}}\fnorm{h'}\,,$
\item $\fnorm{f_h''}\leq 2\fnorm{h'}\,.$
\end{enumerate}
\end{enumerate}
\end{lemma}
\begin{proof}
Assertion (a) of (i) is proved in Section \ref{proofs}. By \eqref{coreq1} $f_h(x)$ coincides with the solution $\ftilde_g(x)$ to the Stein equation for the standard normal distribution corresponding to the test function $g(x)=h(\abs{x})$. 
Since 
\[\sup_{x\in\R}\abs{g(x)-E[g(Z)]}=\sup_{x\geq0}\abs{h(x)-E[h(Y)]}\,,\quad E[g(Z)]=E[h(Y)]\quad\text{and,}\] 
by \eqref{coreq2}, $\fnorm{g'}=\fnorm{h'}$ bound (b) of (i) and  
the bounds in (ii) follow from well-known bounds in the standard normal case (see \cite{CGS})).\\ 
\end{proof}

Lipschitz test functions $h$ with Lipschitz constant $\fnorm{h'}\leq1$ yield the Wasserstein distance. The Kolmogorov distance, in contrast, is induced by the class of functions $h_z=1_{(-\infty,z]}$, $z\in\R$. Since we only compare distributions with support on $[0,\infty)$ we may restrict ourselves to the case $z\geq0$. 
We write $f_z:=f_{h_z}$. Since the functions $h_z$ are uniformly bounded by $1$, we immediately get bounds on $\fnorm{f_z}$ and $\fnorm{f_z'}$ from 
Lemma \ref{bounds} (i). But for this particular class of functions, a special analysis indeed yields smaller constants. If we denote by $\ftilde_z$ the standard solution to Stein's equation for the standard normal distribution with respect to the test function $h_z$, where $z\in\R$, we obtain from \eqref{coreq1} for $z\geq0$ that 
$f_z(x)=\ftilde_z(x)-\ftilde_{-z}(x)$ for each $x\geq0$. Indeed, letting $g_z(x):=h_z(\abs{x})$, $x\in\R$, we have $g_z(x)=h_z(x)-1_{(-\infty,-z)}(x)$ which equals $h_z(x)-h_{-z}(x)$ for every $x\not=-z$, leading to this representation of $f_z$. 
However, we were not able to use this representation and known properties of the functions $\ftilde_z$ to derive the properties of $f_z$ stated in the following lemma. This is why a complete proof is given in Section \ref{proofs}.
\begin{lemma}\label{bounds2}
For each $x,y,z\geq0$ we have:
\begin{enumerate}[{\normalfont (a)}]
\item The function $[0,\infty)\ni x\mapsto xf_z(x)\in\R$ is increasing and\\ $0\leq xf_z(x)\leq2\Phi(z)-1<1$.
\item $0< f_z(x)\leq\frac{1}{2}\,.$\label{fzbound}
\item $\fnorm{f_z'}\leq1$ and $\abs{f_z'(x)-f_z'(y)}\leq1\,.$
\end{enumerate}
\end{lemma}
 
\begin{remark}\label{fzrem}
The bound in Lemma \ref{bounds2} (b) is not optimal. More precisely, computer algebra systems suggest that 
\[\sup_{z\geq0}\fnorm{f_z}=0.456296...\,.\]
However, the bounds in Lemma \ref{bounds2} (c) are optimal, which is proved in Section \ref{proofs}.
\end{remark}

In the following, we review the technique of finding a suitable Stein type characterization for a discrete distribution on the integers by Goldstein and Reinert 
\cite{GolRei13}. 
A finite integer interval is a set $I$ of the form $I=[a,b]\cap\Z$ for some integers $a\leq b$. Given a probability mass function 
$p:\Z\rightarrow\R$ with $p(k)>0$ for $k\in I$ and $p(k)=0$ for $k\in\Z\setminus I$, we consider the function 
$\psi:I\rightarrow\R$ given by the formula
\begin{equation}\label{formelpsi}
\psi(k):=\frac{\Delta p(k)}{p(k)}=\frac{p(k+1)-p(k)}{p(k)},
\end{equation} 
where for a function $f$ on the integers $\Delta f(k):=f(k+1)-f(k)$ denotes the \textit{forward difference operator}. 
The next result, a version of Corollary 2.1 from \cite{GolRei13}, yields various Stein characterizations for the distribution corresponding to 
$p$. For such a probability mass function $p$ with support a finite integer interval $I=[a,b]\cap\Z$, let $\F(p)$ denote the class of all real-valued functions $f$ on $\Z$ such that $f(a-1)=0$.

\begin{proposition}\label{golreingen}
Let $p$ be a probability mass function which is supported on the finite integer interval 
$I=[a,b]\cap\Z$ and is positive there. Let $c:[a-1,b]\cap\Z\rightarrow\R$ be a function with $c(k)\not=0$ for all $k\in I$. Then, in order that a given random variable $X$ with support $I$ is distributed according to $p$ it is necessary and sufficient that for all functions $g\in\F(p)$ we have
\begin{equation}\label{steinidgen}
E\Bigl[c(X-1)\Delta g(X-1)+\bigl[c(X)\psi(X)+\Delta c(X-1)\bigr]g(X)\Bigr]=0\,.
\end{equation}
\end{proposition}

\section{The number of returns to the origin}\label{returns}

Recall the definition of $K_n$ from \eqref{Kn}. In this section we give the proof of Theorem \ref{theo2}. It is known (see e.g. \cite{Fel1}, Problem 9 in Chapter 3) that 
for each $r\in\{0,1,\dotsc,m\}$
\begin{equation}\label{pKn}
 p(r):=P(K_n=r)=\frac{1}{2^{n-r}}\binom{n-r}{n/2}=\frac{1}{2^{2m-r}}\binom{2m-r}{m}\,.
\end{equation}
Using \eqref{pKn} as well as the relation 
\begin{equation*}
 \binom{n+1}{k}=\frac{n+1}{n-k+1}\binom{n}{k}
\end{equation*}
we obtain that 
\begin{align}\label{psiKn}
 \psi(r)&:=\frac{p(r+1)-p(r)}{p(r)}=\frac{2^{-(2m-r-1)}\binom{2m-r-1}{m}-2^{-(2m-r)}\binom{2m-r}{m}}{2^{-(2m-r)}\binom{2m-r}{m}}\notag\\
&=2\frac{m-r}{2m-r}-1=\frac{-r}{2m-r}
\end{align}
for all $r\in\{0,1,\dotsc,m\}$. We thus define $c(r):=2m-r$ for $r=-1,0,\dotsc,m$ and obtain 
\begin{itemize}
 \item $\Delta c(r-1)=c(r)-c(r-1)=2m-r-(2m-r+1)=-1$
 \item $c(r)\psi(r)=-r$ and 
 \item $\gamma(r):=c(r)\psi(r)+\Delta c(r-1)=-r-1=-(r+1)$\,.
\end{itemize}
Proposition \ref{golreingen} thus yields the following characterization of the distribution of $K_n$.
\begin{lemma}\label{Knle}
 A random variable $X$ with support $[0,m]\cap\Z$ has probability mass function $p$ if and only if for every function $g\in\F(p)$ 
\[E\bigl[(2m-X+1)\Delta g(X-1)-(X+1)g(X)\bigr]=0\,.\]
\end{lemma}
 
Letting $g(k):=1$ for $k\geq0$ and $g(k):=0$ for $k<0$ we obtain that 
\begin{align}\label{EKn}
 E[K_n]&=(2m+1)P(K_n=0)-1\notag\\
&=(2m+1)2^{-2m}\binom{2m}{m}-1\leq 2m 2^{-2m}\binom{2m}{m}\notag\\
&\leq\frac{2m}{\sqrt{\pi m}}=\sqrt{\frac{2}{\pi}}\sqrt{n}\,,
\end{align}
where the last inequality follows from Stirling's formula.

\begin{remark}
It is not easy to obtain the above formula for $E[K_n]$ directly from the definition of the expected value, but some combinatorial tricks are needed. Also note that bound \eqref{EKn} on $E[K_n]$ is quite accurate.
Furthermore, since $K_n=\sum_{j=1}^m 1_{A_j}$, where $A_j:=\{S_{2j}=0\}$, we have $E[K_n]=\sum_{j=1}^m2^{-2j}\binom{2j}{j}$ and this is the partial sum of the diverging series 
which is often used to prove recurrence of the simple symmetric random walk. Thus, we see how the Stein characterization in Lemma \ref{Knle} gives us information about how fast this series diverges. This indicates that a lot of information might be encoded in such a Stein identity. 
\end{remark}

\begin{proof}[Proof of Theorem \ref{theo2}]
Let $h$ be a Borel-measurable test function and consider the corresponding standard solution $f:=f_h$ to Stein's equation \eqref{steineq} given by \eqref{fh1} and \eqref{fh2}. We also let $f_h(x):=0$ for each $x<0$ and define $g(k):=f_h(k/\sqrt{n})$ for $k\in\Z$. Writing $\Delta_y f(x):=f(x+y)-f(x)$ we obtain from Lemma \ref{Knle} 
\begin{align}\label{Kneq1}
 E\bigl[Wf(W)\bigr]&=\frac{1}{\sqrt{n}}E\bigl[K_n g(K_n)\bigr]
=\frac{1}{\sqrt{n}} E\bigl[(n-K_n+1)\Delta g(K_n-1)-g(K_n)\bigr]\notag\\
&=E\Bigl[\bigl(\sqrt{n}-W+n^{-1/2}\bigr)\Delta_{n^{-1/2}} f\bigl(W-n^{-1/2}\bigr)\Bigr]-\frac{1}{\sqrt{n}}E\bigl[f(W)\bigr]\notag\\
&=\sqrt{n}E\bigl[\Delta_{n^{-1/2}} f\bigl(W-n^{-1/2}\bigr)\bigr]+E_1\,,
\end{align}
where 
\begin{align}\label{Kneq2}
 \abs{E_1}&=\Bigl|E\Bigl[\bigl(n^{-1/2}-W\bigr)\Delta_{n^{-1/2}} f\bigl(W-n^{-1/2}\bigr)\Bigr]-\frac{1}{\sqrt{n}}E\bigl[f(W)\bigr]\Bigr|\notag\\
&\leq \fnorm{f'}\left(\frac{1}{n}+\frac{1}{\sqrt{n}}E[W]\right)+\frac{\fnorm{f}}{\sqrt{n}}\notag\\
&\leq\frac{\fnorm{f'}}{\sqrt{n}}\left(\frac{1}{\sqrt{n}}+\sqrt{\frac{2}{\pi}}\right) +\frac{\fnorm{f}}{\sqrt{n}} 
\end{align}
by inequality \eqref{EKn} and because $W=n^{-1/2}K_n$.
Now, using \eqref{Kneq1} and the fact that $f=f_h$ is a solution to the half-normal Stein equation \eqref{steineq} we obtain
\begin{align}\label{Kneq3}
\bigl|E[h(W)]-E[h(Y)]\bigr|&=\bigl|E\bigl[f'(W)-Wf(W)\bigr]\bigr|\notag\\
&\leq \bigl|E\bigl[f'(W)-\sqrt{n}\Delta_{n^{-1/2}} f\bigl(W-n^{-1/2}\bigr)\bigr]\bigr|+\abs{E_1}\notag\\
&=\sqrt{n}\Bigl|E\Bigl[\int_{W-n^{-1/2}}^W\bigl(f'(W)-f'(t)\bigr)dt\Bigr]\Bigr|+\abs{E_1}\notag\\
&=\abs{E_2}+\abs{E_1}\,,
\end{align}
where, 
\begin{equation}\label{Kneq4}
E_2:=\sqrt{n}E\Bigl[\int_{W-n^{-1/2}}^W\bigl(f'(W)-f'(t)\bigr)dt\Bigr]\,.
\end{equation}
If $h$ is Lipschitz, then we know from Lemma \ref{bounds} (ii) that $f'$ is also Lipschitz with Lipschitz constant $2\fnorm{h'}$ and we can bound 
\begin{align}\label{Kneq5}
 \abs{E_2}&\leq\sqrt{n}\fnorm{f''}E\Bigl[\int_{W-n^{-1/2}}^W\bigl(W-t\bigr)dt\Bigr]=\sqrt{n}\fnorm{f''}\int_{0}^{n^{-1/2}}u du\notag\\
&=\sqrt{n}\fnorm{f''}\frac{1}{2n}=\frac{\fnorm{f''}}{2\sqrt{n}}\leq\frac{\fnorm{h'}}{\sqrt{n}}\,.
\end{align}
More generally, if $h$ is measurable and $E\abs{h(Y)}<\infty$, then, again using that $f=f_h$ is a solution to \eqref{steineq} we conclude 
\begin{align}\label{Kneq6}
E_2&=\sqrt{n}E\Bigl[\int_{W-n^{-1/2}}^W\bigl(Wf(W)-tf(t)+h(W)-h(t)\bigr)dt\Bigr]\notag\\
&= \sqrt{n} E\Bigl[\int_{W-n^{-1/2}}^W\bigl(Wf(W)-tf(t)\bigr)dt\Bigr] +\sqrt{n}E\Bigl[\int_{W-n^{-1/2}}^W\bigl(h(W)-h(t)\bigr)dt\Bigr]\notag\\
&=:E_3+E_4\,.
\end{align}
By the fundamental theorem of calculus we have 
\begin{align}\label{Kneq14}
 \abs{E_3}&=\sqrt{n}\Bigl|E\Bigl[\int_{W-n^{-1/2}}^W\int_t^W(f(s)+sf'(s))ds\,dt\Bigr]\Bigr|\notag\\
&\leq\sqrt{n}\fnorm{f}E\Bigl[\int_{W-n^{-1/2}}^W\bigl(W-t\bigr)dt\Bigr]+\sqrt{n}\fnorm{f'}E\Bigl[\int_{W-n^{-1/2}}^W\int_t^W\abs{s}ds\,dt\Bigr]\notag\\
&\leq\sqrt{n}\fnorm{f}\frac{1}{2n}+\sqrt{n}\fnorm{f'}E\Bigl[\max(n^{-1/2}, W)\int_{W-n^{-1/2}}^W\bigl(W-t\bigr)dt\Bigr]\notag\\
&=\frac{\fnorm{f}}{2\sqrt{n}}+\frac{\fnorm{f'}}{2\sqrt{n}}E\bigl[\max(n^{-1/2}, W)\bigr]\notag\\
&\leq\frac{\fnorm{f}}{2\sqrt{n}}+\frac{\fnorm{f'}}{2n}+\frac{\fnorm{f'}}{2\sqrt{n}}\sqrt{\frac{2}{\pi}}\,,
\end{align}
where we have used the inequality $\max(x,y)\leq x+y$ valid for $x,y\geq0$ and \eqref{EKn} for the last step.
Note that for Lipschitz $h$ we have $|h(W)-h(t)|\leq\fnorm{h'}|W-t|$ and hence 
\[\abs{E_4}\leq \sqrt{n}\fnorm{h'}E\Bigl[\int_{W-n^{-1/2}}^W\bigl(W-t\bigr)dt\Bigr]=\frac{\fnorm{h'}}{2\sqrt{n}}\,,\]
which together with \eqref{Kneq14} yields a worse bound than the one obtained in \eqref{Kneq5}, if we plug in the bounds on $f$ and $f'$ 
from Lemma \ref{bounds} (ii). 
If $h$ is not Lipschitz, then $\abs{E_4}$ cannot be bounded that easily. Having in mind the Kolmogorov distance, we restrict ourselves to the test functions 
$h_z=1_{(-\infty,z]}$, $z>0$. Note that for $h=h_z$ we can write 
\begin{equation*}
 \abs{E_4}=-E_4=\sqrt{n}E\Bigl[\int_\R 1_{\{W-n^{-1/2}\leq t\leq W\}}\bigl(h(t)-h(W)\bigr)dt\Bigr] 
\end{equation*}
and that 
\begin{align}\label{Kneq8}
& 1_{\{W-n^{-1/2}\leq t\leq W\}}\bigl(h(t)-h(W)\bigr)=1_{\{W-n^{-1/2}\leq t\leq W\}} 1_{\{W-n^{-1/2}\leq t\leq z< W\}}\notag\\
&=1_{\{W-n^{-1/2}\leq t\leq z\}} 1_{\{W-n^{-1/2}\leq z< W\}}\notag\\
&\leq 1_{\{W-n^{-1/2}\leq t\leq W\}} 1_{\{W-n^{-1/2}\leq z< W\}}\,.
\end{align}
Thus, using \eqref{Kneq8} we obtain
\begin{align}\label{Kneq9}
 \abs{E_4}&=\sqrt{n}E\Bigl[\int_\R 1_{\{W-n^{-1/2}\leq t\leq z\}} 1_{\{W-n^{-1/2}\leq z< W\}}dt\Bigr]\notag\\
&\leq \sqrt{n}E\Bigl[\int_{W-n^{-1/2}}^W 1 dt 1_{\{W-n^{-1/2}\leq z< W\}}\Bigr]\notag\\
&=P\bigl(W-n^{-1/2}\leq z< W\bigr)=P\bigl(K_n-1\leq\sqrt{n}z<K_n\bigr)\notag\\
&=P\bigl(K_n=1+\lfloor \sqrt{n}z\rfloor\bigr)\,.
\end{align}
Now note that by unimodality of binomial coefficients for each $r=0,1,\dotsc,m$ we can bound 
\begin{align}\label{Kneq10}
P(K_n=r)&=\frac{1}{2^{2m-r}}\binom{2m-r}{m}\leq\frac{1}{2^{2m-r}}
\begin{cases}
\binom{2m-r}{m-r/2},& r \text{  even}\\
\binom{2m-r}{m-(r+1)/2},& r\text{  odd}        
\end{cases}\notag\\
&\leq
\begin{cases}
 (\pi(m-r/2))^{-1/2},& r\text{  even}\\
 (\pi(m-(r+1)/2))^{-1/2},& r\text{  odd} 
\end{cases}\notag\\
&\leq\frac{\sqrt{2}}{\sqrt{\pi m}}\,,
\end{align}
where the next to last inequality comes from 
\[2^{-2k-1}\binom{2k+1}{k}\leq 2^{-2k}\binom{2k}{k}\leq\frac{1}{\sqrt{\pi k}}\]
by Stirling's formula.
Note that $P(K_n=r)=0$ for $r>m=n/2$ and, hence, $P(K_n=1+\lfloor \sqrt{n}z\rfloor)=0$ unless $z\leq 2^{-1}n^{1/2}-n^{-1/2}$. 
Thus, from \eqref{Kneq10} we have that for each $z>0$ 
\begin{equation}\label{Kneq11}
\abs{E_4}\leq\frac{\sqrt{2}}{\sqrt{\pi m}}=\frac{2}{\sqrt{\pi n}}\,.
\end{equation}
Collecting terms we see from \eqref{Kneq3}, \eqref{Kneq2}, \eqref{Kneq5} and Lemma \ref{bounds} (ii) that for $h$ Lipschitz on $[0,\sqrt{n}]$ we have  
\begin{align}\label{Kneq12}
 \bigl|E[h(W)]-E[h(Y)]\bigr|\leq\frac{\fnorm{h'}}{\sqrt{n}}\left(\frac{1}{\sqrt{n}}\sqrt{\frac{2}{\pi}}+\frac{2}{\pi}+2\right)
\end{align}
 and letting $h=h_z$ for $z\geq0$ we see from \eqref{Kneq3}, \eqref{Kneq2}, \eqref{Kneq6}, \eqref{Kneq14}, \eqref{Kneq11} and Lemma \ref{bounds2} that 
\begin{align}\label{Kneq13}
 \bigl|P(W\leq z)-P(Y\leq z)\bigr|\leq\frac{1}{\sqrt{n}}\left(\frac{3+2\sqrt{2}}{\sqrt{2\pi}}+\frac{3}{4}\right)+\frac{3}{2n}\,.
\end{align}
Theorem \ref{theo2} now follows from \eqref{Kneq12} and \eqref{Kneq13}.\\
\end{proof}

\section{The maximum of the simple random walk}\label{max}
Recall the definition of $M_n$ from \eqref{Mn}, where $n=2m$ is an even positive integer. In this section we give a proof of Theorem \ref{theo1}, which is a little bit more 
complicated than the one of Theorem \ref{theo2} given in the last section. This is so because in this case the function $\psi$ from \eqref{formelpsi} vanishes on the odd 
integers which is inconvenient for our method of proof. So we will introduce an auxiliary variable $V$ and use the triangle inequality 
\begin{equation}\label{triangle}
 d_\mathcal{H}(W,Y)\leq d_\mathcal{H}(W,V)+d_\mathcal{H}(V,Y)
\end{equation}
to prove the theorem.\\
For $r=0,1,\dotsc,n=2m$ let $p(r):=P(M_n=r)$. Then, it is known (see \cite{Fel1}, Theorem 1 in Section 8 of Chapter 3) that 
\begin{equation}\label{Mneq1}
 p(r)=P(M_n=r)=p_{n,r}+p_{n,r+1}\,,
\end{equation}
where 
\begin{equation*}
p_{n,k}:=P(S_n=k)=\binom{n}{\frac{n+k}{2}}2^{-n} 
\end{equation*}
is the probability that the random walk is in position $k$ at time $n$. We use the convention that $\binom{y}{x}=0$ unless $x$ is a nonnegative integer. Hence, if 
$p_{n,k}\not=0$, then $n$ and $k$ must have the same parity and since we always assume that $n$ is even, we have $p_{n,k}=0$ whenever $k$ is odd.

\begin{remark}\label{remclt}
From \eqref{Mneq1} one can easily get that for $z\geq0$
\begin{align}\label{clt1}
\Bigl|P(M_n\leq\sqrt{n}z)-(2\Phi(z)-1)\Bigr|&=\Bigl|2\bigl(P(S_n\leq z)-\Phi(z)\bigr)+P(S_n=1+\lfloor\sqrt{n}z\rfloor)\Bigr|\,, 
\end{align}
which, together with the Berry-Esseen Theorem for Bernoulli random variables, yields the second part of Theorem \ref{theo1}, maybe with different constants. Using the fact that 
\begin{equation*}
 d_{\W}(\mu,\nu)=\int_\R\bigl|F(x)-G(x)\bigr|dx\,,
\end{equation*}
where $F$ and $G$ are the distribution functions corresponding to the distributions $\mu$ and $\nu$, respectively, one might also be able to derive a bound comparable to the first bound of Theorem \ref{theo1} from \eqref{clt1} 
and a quantitative version of the mean CLT, see \cite{Gol10}, for instance. However, we prefer giving a full proof of Theorem \ref{theo1} which does not rely on any external quantitative CLT results but is just based on the Stein characterizations of the two distributions involved. 
\end{remark}

From \eqref{Mneq1} 
it easily follows that with $\psi$ given by \eqref{formelpsi} for all $r\in\{0,1,\dotsc,n\}$ 
\begin{align}\label{Mneq2}
 \psi(r)&:= \frac{p_{n,r+2}-p_{n,r}}{p_{n,r}+p_{n,r+1}}=
\begin{cases}
 \frac{p_{n,r+2}-p_{n,r}}{p_{n,r}},& r\text{  even}\\
0,& r\text{  odd}\,.
\end{cases}
\end{align}
This is why we introduce the auxiliary random variables 
\begin{equation}\label{defNn}
N_n:=\left\lfloor\frac{M_n+1}{2}\right\rfloor\quad\text{and}\quad V:=V_n:=\frac{2N_n}{\sqrt{n}}\,.
\end{equation}
Then, $N_n$ only takes the values $0,1,\dotsc,m$ and for $s\in\{0,1,\dotsc,m\}$ we obtain from \eqref{Mneq1} that 
\begin{align}\label{Mneq3}
 q(s)&:=P(N_n=s)=P(M_n+1=2s)+P(M_n+1=2s+1)\notag\\
&=P(M_n=2s-1)+P(M_n=2s)=2P(M_n=2s)=2p(2s)\notag\\
&=2\binom{2m}{m+s}2^{-2m}\,.
\end{align}
Denoting by $\rho$ the corresponding difference quotient of the probability mass function $q$ we obtain for $s=0,1,\dotsc,m$ that 
\begin{align}\label{Mneq4}
 \rho(s)&:=\frac{q(s+1)-q(s)}{q(s)}=\frac{2p(2s+2)-2p(2s)}{2p(2s)}\notag\\
&=\frac{p(2s+2)-p(2s)}{p(2s)}=\frac{\binom{2m}{m+s+1}-\binom{2m}{m+s}}{\binom{2m}{m+s}}\notag\\
&=\frac{\binom{2m}{m+s}\bigl(\frac{m-s}{m+s+1}-1\bigr)}{\binom{2m}{m+s}}=\frac{m-s}{m+s+1}-\frac{m+s+1}{m+s+1}\notag\\
&=-\frac{2s+1}{m+s+1}\,.
\end{align}
So, we define the function $c:[-1,m]\rightarrow\R$ by $c(s):=m+s+1$. Then, for $s=0,1,\dotsc,m$
\begin{align*}
c(s)\rho(s)&=-2s-1=-(2s+1)\\
\Delta c(s-1)&=c(s)-c(s-1)=m+s+1-(m+s)=1\\
\gamma(s)&:=c(s)\rho(s)+\Delta c(s-1)=-2s-1+1=-2s\,. 
\end{align*}
Hence, Proposition \ref{golreingen} implies the following lemma.
\begin{lemma}\label{Mnlemma}
 A random variable $X$ with values in $I:=[0,m]\cap\Z$ has the probability mass function $q$, if and only if for all functions $g\in\F(q)$ 
\[E\bigl[(m+X)\Delta g(X-1)-2Xg(X)\bigr]=0\,.\]
\end{lemma}

Letting $g(k):=1$ for $k\geq0$ and $g(k):=0$ for $k<0$, we see from Lemma \ref{Mnlemma} that 
\begin{align}\label{ENn}
 E[N_n]&=\frac{m P(N_n=0)}{2}=\frac{2m2^{-2m}\binom{2m}{m}}{2}=m2^{-2m}\binom{2m}{m}\notag\\
&\leq\sqrt{\frac{m}{\pi}}=\frac{1}{\sqrt{2\pi}}\sqrt{n}\notag\\
&\Rightarrow E[V]=\frac{2}{\sqrt{n}}E[N_n]\leq\sqrt{\frac{2}{\pi}}\,.
\end{align}
The next lemma gives bounds on the distance from $V$ to $W$.
\begin{lemma}\label{VWlemma}
 For each $n=2m$ we have 
\begin{align*}
d_\W(V,W)&\leq\frac{1}{\sqrt{n}}\quad\text{and}\quad d_\K(V,W)\leq\sqrt{\frac{2}{\pi}}\frac{1}{\sqrt{n}}\,. 
\end{align*}
\end{lemma}
\begin{proof}
We have 
\begin{align}\label{Mneq5}
&\frac{M_n-1}{2}=\frac{M_n+1}{2}-1<N_n\leq\frac{M_n+1}{2}\notag\\
\Rightarrow &-1<2N_n-M_n\leq1\notag\\
\Rightarrow&-\frac{1}{\sqrt{n}}<V-W\leq\frac{1}{\sqrt{n}}\,.
\end{align}
Hence, if $h$ is Lipschitz on $[0,\sqrt{n}]$ with constant $1$, then from \eqref{Mneq5} follows that
\begin{equation*}
 \Bigl|E\bigl[h(V)-h(W)\bigr]\Bigr|\leq\fnorm{h'}E\abs{V-W}\leq\frac{1}{\sqrt{n}}\,,
\end{equation*}
which proves the first claim. As to the second claim, note that the Kolmogorov distance is scale-invariant which implies that 
$d_\K(V,W)=d_\K(2N_n,M_n)$ and that for $z\in[0,\infty)$
\[\bigl|P(2N_n\leq z)-P(M_n\leq z)\bigr|=\bigl|P(2N_n\leq\lfloor z\rfloor)-P(M_n\leq\lfloor z\rfloor)\bigr|\,.\]
Hence, we need only consider $z\in\N_0$. If $z=2k$ is even, then 
\begin{equation*}
\{2N_n\leq 2k\}=\{\bigl\lfloor\frac{M_n+1}{2}\bigr\rfloor\leq k\}=\{M_n\leq 2k\} 
\end{equation*}
and, thus, 
\begin{equation}\label{Mneq6}
 P(2N_n\leq 2k)=P(M_n\leq 2k)\,.
\end{equation}
On the other hand, if $z=2k+1$, then 
\begin{equation*}
 \{2N_n\leq 2k+1\}=\{2N_n\leq 2k\}=\{M_n\leq 2k\}\subseteq\{M_n\leq2k+1\}
\end{equation*}
and 
\begin{equation*}
 \{M_n\leq2k+1\}\setminus\{2N_n\leq 2k+1\}=\{M_n=2k+1\}\,.
\end{equation*}
Hence, 
\begin{align}\label{Mneq7}
 \bigl|P(2N_n\leq 2k+1)-P(M_n\leq 2k+1)\bigr|&=P(M_n=2k+1)=p_{n,2k+2}\notag\\
&=\binom{2m}{m+k+1}2^{-2m}\leq\binom{2m}{m}2^{-2m}\notag\\
&\leq\frac{1}{\sqrt{\pi m}}=\frac{\sqrt{2}}{\sqrt{\pi n}}\,,
\end{align}
by Stirling's formula. Thus, for every $z\in\R$ we have the bound 
\[\bigl|P(2N_n\leq z)-P(M_n\leq z)\bigr|\leq\sqrt{\frac{2}{\pi }}\frac{1}{\sqrt{n}}\,,\]
proving the second claim of the lemma. \\
\end{proof}
Having bounded the distance from $V$ to $W$, we may now derive bounds on the distance from $V$ to $Y$ in a similar way as in Section \ref{returns}.
\begin{lemma}\label{VYlemma}
 For the distance from $V$ to $Y$ we have
\begin{align*}
 d_\W(V,Y)&\leq\frac{1}{\sqrt{n}}\Bigl(2+\frac{4}{\pi}\Bigr)+2\sqrt{\frac{2}{\pi}}\frac{1}{n}\quad\text{and}\\
 d_\K(V,Y)&\leq\frac{1}{\sqrt{n}}\Bigl(3\sqrt{\frac{2}{\pi}}+\frac{1}{2}\Bigr)+\frac{2}{n}\,.
\end{align*}
\end{lemma}
\begin{proof}
Let $h$ be a Borel-measurable test function and consider the corresponding standard solution $f:=f_h$ to Stein's equation \eqref{steineq} given by 
\eqref{fh1} and \eqref{fh2}. We also let $f_h(x):=0$ for each $x<0$ and define $g(k):=f_h(2k/\sqrt{n})$ for $k\in\Z$. Writing $\Delta_y f(x):=f(x+y)-f(x)$ we 
obtain from Lemma \ref{Mnlemma} 
\begin{align}\label{Mneq8}
 E\bigl[Vf(V)\bigr]&=\frac{2}{\sqrt{n}}E\bigl[N_n g(N_n)\bigr]
=\frac{1}{\sqrt{n}} E\bigl[(m+N_n)\Delta g(N_n-1)\bigr]\notag\\
&=E\Bigl[\Bigl(\frac{\sqrt{n}}{2}+\frac{V}{2}\Bigr)\Delta_{2n^{-1/2}} f\bigl(V-2n^{-1/2}\bigr)\Bigr]\,.
\end{align}
Thus, since $f$ solves Stein's equation \eqref{steineq} we obtain from \eqref{Mneq8} 
\begin{align}\label{Mneq9}
\Bigl|E\bigl[h(V)\bigr]-E\bigl[h(Y)\bigr]\Bigr|&=\Bigl|E\bigl[f'(V)-Vf(V)\bigr]\Bigr|\notag\\
&=\Bigl|E\Bigl[f'(V)-\Bigl(\frac{\sqrt{n}}{2}+\frac{V}{2}\Bigr)\Delta_{2n^{-1/2}} f\bigl(V-2n^{-1/2}\bigr)\Bigr]\Bigr|\notag\\
&=\Bigl|E\Bigl[f'(V)-\frac{\sqrt{n}}{2}\Delta_{2n^{-1/2}} f\bigl(V-2n^{-1/2}\bigr)\Bigr]-E_1\Bigr|
\end{align}
where 
\begin{align}\label{Mneq10}
 \abs{E_1}&=\Bigl|E\Bigl[\frac{V}{2}\Delta_{2n^{-1/2}} f\bigl(V-2n^{-1/2}\bigr)\Bigr]\Bigr|\leq\frac{2n^{-1/2}}{2}\fnorm{f'}E[V]\notag\\
&\leq\fnorm{f'}\sqrt{\frac{2}{\pi}}\frac{1}{\sqrt{n}}
\end{align}
by inequality \eqref{ENn}. Similarly to the proof of Theorem \ref{theo2} in Section \ref{returns} we have
\begin{align}\label{Mneq11}
 \abs{E_2}&:=\Bigl|E\Bigl[f'(V)-\frac{\sqrt{n}}{2}\Delta_{2n^{-1/2}} f\bigl(V-2n^{-1/2}\bigr)\Bigr]\Bigr|\notag\\
&=\frac{\sqrt{n}}{2}\Bigl|E\Bigl[\int_{V-\frac{2}{\sqrt{n}}}^V\bigl(f'(V)-f'(t)\bigr)dt\Bigr]\Bigr|\notag\\
&=\frac{\sqrt{n}}{2}\Bigl|E\Bigl[\int_{V-\frac{2}{\sqrt{n}}}^V\bigl(Vf(V)-tf(t)+h(V)-h(t)\bigr)dt\Bigr]\Bigr|\notag\\
&\leq\frac{\sqrt{n}}{2}\Bigl|E\Bigl[\int_{V-\frac{2}{\sqrt{n}}}^V\bigl(Vf(V)-tf(t)\bigl)dt\Bigr]\Bigr|
+\frac{\sqrt{n}}{2}\Bigl|E\Bigl[\int_{V-\frac{2}{\sqrt{n}}}^V\bigl(h(V)-h(t)\bigr)dt\Bigr]\Bigr|\notag\\
&=:\abs{E_3}+\abs{E_4}\,.
\end{align}
 As for \eqref{Kneq5}, if $h$ is Lipschitz than by bound (c) from Lemma \ref{bounds} (ii) we easily get 
\begin{equation}\label{Mneqplus2}
 \abs{E_2}\leq\fnorm{f''}\frac{\sqrt{n}}{2}\int_{V-\frac{2}{\sqrt{n}}}^V (V-t)dt=\frac{\sqrt{n}}{2}\fnorm{f''}\frac{2}{n}\leq\frac{2\fnorm{h'}}{\sqrt{n}}
\end{equation}
Similar computations as those in Section \ref{returns} yield 
\begin{align}\label{Mneq12}
 \abs{E_3}&\leq\frac{1}{\sqrt{n}}\left(\fnorm{f}+\fnorm{f'}\bigl(\frac{2}{\sqrt{n}}+E[V]\bigr)\right)\\
&\leq\frac{1}{\sqrt{n}}\biggl(\fnorm{f}+\fnorm{f'}\Bigl(\frac{2}{\sqrt{n}}+\sqrt{\frac{2}{\pi}}\Bigr)\biggl)\notag
\end{align}
and 
\begin{equation}\label{Mneq13}
\abs{E_4}\leq\frac{\fnorm{h'}}{\sqrt{n}} 
\end{equation}
for Lipschitz continuous functions $h$. With a similar computation as the one leading to \eqref{Kneq9} one can show that for $h=h_z$, where $z\geq0$, 
\begin{equation}\label{Mneqplus}
 \abs{E_4}\leq P\Bigl(N_n=\Bigl\lfloor \frac{\sqrt{n}z}{2}\Bigr\rfloor+1\Bigr)\leq\max_{s=0,\dotsc,m}q(s)\,.
\end{equation}
But from \eqref{Mneq3} we see that 
\begin{equation}\label{Mneq14}
 \max_{s=0,\dotsc,m}q(s)=q(0)\leq\frac{2}{\sqrt{\pi m}}=\sqrt{\frac{2}{\pi}}\frac{2}{\sqrt{ n}}\,,
\end{equation}
again by Stirling's formula.
From \eqref{Mneq13} and \eqref{Mneq14} we get that for each $z\in\R$
\begin{equation}\label{Mneq15}
 \abs{E_4}\leq\sqrt{\frac{2}{\pi}}\frac{2}{\sqrt{ n}}\,.
\end{equation}
Collecting terms, we see from \eqref{Mneq9}, \eqref{Mneq10}, \eqref{Mneqplus2} and Lemma \ref{bounds} (ii) that for $h$ Lipschitz on 
$[0,\sqrt{n}]$ we have 
\begin{equation}\label{Mneq16}
\Bigl|E\bigl[h(V)\bigr]-E\bigl[h(Y)\bigr]\Bigr|\leq \fnorm{h'}\biggl(\frac{1}{\sqrt{n}}\Bigl(2+\frac{2}{\pi}\Bigr)\biggr)
\end{equation}
and from \eqref{Mneq9}, \eqref{Mneq10}, \eqref{Mneq11}, \eqref{Mneq12}, \eqref{Mneq15} and Lemma \ref{bounds2} we see that for each $z\in\R$ we have the bound 
\begin{equation}\label{Mneq17}
 \Bigl|P(V\leq z)-P(Y\leq z)\Bigr|\leq \frac{1}{\sqrt{n}}\Bigl(3\sqrt{\frac{2}{\pi}}+\frac{1}{2}\Bigr)+\frac{2}{n}\,.
\end{equation}
The claims of the lemma now follow from \eqref{Mneq16} and \eqref{Mneq17}, respectively.\\
\end{proof}
\begin{proof}[Proof of Theorem \ref{theo1}]
The claims of Theorem \ref{theo1} now follow easily from Lemmas \ref{VWlemma}, \ref{VYlemma} and from the triangle inequality \eqref{triangle}.\\
\end{proof}

\section{The number of sign changes}\label{signs}
Recall the definition of $C_{2m+1}$ from \eqref{Cn}. It is known (see Theorem 1 in Section 5 of Chapter 3 of \cite{Fel1}, for example) that for $s=0,1,\dotsc,m$ 
\begin{equation}\label{Cneq1}
 p(s):=P(C_{2m+1}=s)=2p_{2m+1,2s+1}=2\binom{2m+1}{m+s+1}2^{-2m-1}\,.
\end{equation}
Thus, $p$ is very similar to the probability mass function $q$ of $N_n$ from Section \ref{max} (see \eqref{Mneq3}). This allows for a proof of Theorem \ref{theo3}, 
which is completely analogous to the proof of Lemma \ref{VYlemma} and an analogue of Remark \ref{remclt} is also valid, here. This is why we omit the details of the proof 
but just give the suitable Stein characterization of $C_{2m+1}$, which easily follows from \eqref{Cneq1} and 
Proposition \ref{golreingen}.
\begin{lemma}\label{Cnlemma}
 A random variable $X$ with values in $I:=[0,m]\cap\Z$ has the probability mass function $p$, if and only if for all functions $g\in\F(p)$ 
\[E\bigl[(m+1+X)\Delta g(X-1)-2(X+1)g(X)\bigr]=0\,.\]
\end{lemma}

\section{Optimality of the rates}\label{rates}
In this section we give an argument why the rate $n^{-1/2}$ in Theorems \ref{theo1}-\ref{theo3} is optimal for both the Kolmogorov and the Wasserstein distance. Let us 
explain this by means of the example of the number of returns $K_n$. To see that the rate is optimal for the Kolmogorov distance, it suffices to take $z=0$ and observe that 
\[\bigl|P(K_n\leq0)-P(Y\leq0)\bigr|=P(K_n=0)=2^{-2m}\binom{2m}{m}\sim\frac{1}{\sqrt{\pi m}}=\sqrt{\frac{2}{\pi}}\frac{1}{\sqrt{n}}\]
by Stirling's formula. Hence, $n^{-1/2}$ is optimal and the best constant is no less than $\sqrt{\frac{2}{\pi}}$. To show that the rate is also optimal with respect to the 
Wasserstein distance, we consider the function $h(x):=x$ which is $1$-Lipschitz on $\R$. By \eqref{EKn} we have 
\begin{align*}
E[W]&=\frac{1}{\sqrt{n}}E[K_n]=\frac{1}{\sqrt{2m}}(2m+1)2^{-2m}\binom{2m}{m}-1\\
&\sim\frac{\sqrt{2m}}{\sqrt{\pi m}}-\frac{1}{\sqrt{n}}+\frac{1}{\sqrt{\pi m}\sqrt{n}}\\
&=\sqrt{\frac{2}{\pi}}-\frac{1}{\sqrt{n}}+\sqrt{\frac{2}{\pi}}\frac{1}{n}\,.
\end{align*}
Since $E[Y]=\sqrt{\frac{2}{\pi}}$ this shows that 
\[\bigl|E[W]-E[Y]\bigr|\sim\frac{1}{\sqrt{n}}\,,\]
yielding the optimality of the rate $n^{-1/2}$. In a similar fashion one may prove that the convergence rates in Theorems \ref{theo1} and \ref{theo3} are best possible.

\section{Proofs from Section \ref{Stein}}\label{proofs}

In this section we give proofs of some of the results from Section \ref{Stein}.  Most of them are quite standard in Stein's method, but in many cases we obtain better constants than we would by quoting the general bounds from the literature, for example from \cite{ChSh}. 
Recall the density function $p$ and the distribution function $F$ of the half-normal distribution from \eqref{pdf} and \eqref{df}, respectively. Also, in this section, we let 

\begin{equation}\label{psi}
\psi(x):=\frac{d}{dx}\log p(x)=\frac{p'(x)}{p(x)}=-x\,,\quad x\geq0
\end{equation}  

denote the logarithmic derivative of the density function $p$. Note that $\psi$ is a decreasing function. This property will suffice to prove more explicit bounds on the solution to Stein's equation in the general density approach than those currently given in the literature (see \cite{CGS} or \cite{ChSh}). We will return to this issue later. 
We will several times make use of the following well-known \textit{Mill's ratio inequality}, which is valid for $x>0$:

\begin{equation}\label{Mill}
\frac{x}{1+x^2}\phi(x)\leq 1-\Phi(x)\leq\frac{\phi(x)}{x}\,.
\end{equation}
 
Also, for a Borel-measurable test function $h:[0,\infty)\rightarrow\R$ such that $E\abs{h(Y)}<\infty$ we let 

\begin{equation*}
\htilde(x):=h(x)-E\bigl[h(Y)\bigr]\,.
\end{equation*}

\begin{proof}[Proof of Lemma \ref{bounds}]
From \eqref{fh1} and \eqref{fh2} one immediately gets
\begin{equation}\label{beq1}
\abs{f_h(x)}\leq\fnorm{\htilde}\min\bigl(M(x),\,N(x)\bigr)\,,
\end{equation}
where 
\begin{equation}\label{beq2}
M(x):=\frac{F(x)}{p(x)}\quad\text{and}\quad N(x):=\frac{1-F(x)}{p(x)}\,.
\end{equation}
We have 
\begin{align}
M'(x)&=\frac{p(x)^2-p'(x)F(x)}{p(x)^2}=\frac{1}{p(x)}\Bigl(p(x)-\psi(x)F(x)\Bigr)\label{Minc}\quad\text{and}\\
N'(x)&=\frac{-1}{p(x)}\Bigl(p(x)+\psi(x)\bigl(1-F(x)\bigr)\Bigr)\label{Ndec}\,.
\end{align}
Note that 
\begin{align*}
p(x)&=p(0)+\int_0^x p'(t)dt=p(0)+\int_0^x\psi(t)p(t)dt\notag\\
&\geq p(0)+\psi(x)\int_0^xp(t)dt=p(0)+\psi(x)F(x)\\
&\geq\psi(x)F(x)\,.
\end{align*}
Hence, from \eqref{Minc} we conclude that $M$ is increasing on $[0,\infty)$. Similarly, one shows that $N$ is decreasing. Since $N(0)>0=M(0)$ and $\lim_{x\to\infty}M(x)>\lim_{x\to\infty}N(x)=0$ there is a unique point $x\geq0$ such that $M(x)=N(x)$ and, clearly, 
at this point the function $\min(M,N)$ attains its maximum value. But 
\[M(x)=N(x)\Leftrightarrow F(x)=1-F(x)\Leftrightarrow F(x)=\frac{1}{2}\]
and, thus, $x=m$ is the median of $F$. In our case of the half-normal distribution we have $F(x)=2\Phi(x)-1$ and so 
\[m=\Phi^{-1}(3/4)=:z_{0.75}\,.\]
Since $p(x)=2\phi(x)$ and $F(m)=1/2$ we obtain from \eqref{beq1} that 
\begin{equation}\label{beq3}
\fnorm{f_h}\leq\frac{\fnorm{\htilde}}{4\phi(z_{0.75})}\,,
\end{equation}
proving the first claim of (i). For the second claim note that since $f_h$ is a solution to the half-normal Stein equation \eqref{steineq}
\begin{equation}\label{beq4}
\abs{f_h'(x)}=\abs{\htilde(x)-\psi(x)f_h(x)}\leq\fnorm{\htilde}+\fnorm{\psi f_h}\,.
\end{equation}
Using \eqref{beq1} and \eqref{Mill} we have that
\begin{align}\label{beq5}
\abs{\psi(x)f_h(x)}&\leq\fnorm{\htilde}\abs{\psi(x)}\min\bigl(M(x),N(x)\bigr)
\leq\fnorm{\htilde}xN(x)\notag\\
&=\fnorm{\htilde}\frac{x(1-\Phi(x))}{\phi(x)}\leq\fnorm{\htilde}\,.
\end{align}
Thus, 
\[\fnorm{\psi f_h}\leq\fnorm{\htilde}\]
and \eqref{beq4} yields the second claim of (i).\\
\end{proof}

\begin{proof}[Proof of Lemma \ref{bounds2}]
For $z,x\geq0$ we have the representation
\begin{align}\label{bseq1}
f_z(x)&=\frac{F(x\wedge z)-F(x)F(z)}{p(x)}\notag\\
&=\begin{cases}
\frac{(1-F(z))F(x)}{p(x)},&\; x\leq z\\
\frac{(1-F(x))F(z)}{p(x)},&\; x> z
\end{cases}
=\begin{cases}
\bigl(1-F(z)\bigr)M(x),&\; x\leq z\\
F(z) N(x),&\; x> z\,.
\end{cases}
\end{align}
In particular, $f_z$ is positive everywhere. To prove (a) note that  
\begin{align}\label{bseq8}
0\leq xf_z(x)=\begin{cases}
          \bigl(1-F(z)\bigr)xM(x),&x\leq z\\
          F(z)xN(x),& x>z  \,.
         \end{cases}
\end{align}
Now, using \eqref{Ndec} we obtain 
\begin{align*}
\frac{d}{dx} xN(x)&=N(x)+xN'(x)=\frac{(1-x^2)\bigl(1-\Phi(x)\bigr)-x\phi(x)}{\phi(x)}\\
&\geq0 
\end{align*}
by \eqref{Mill}. Thus, from \eqref{bseq8} we see that $xf_z(x)$ is increasing on $[z,\infty)$. 
Similarly, using \eqref{Minc} we have 
\begin{align*}
\frac{d}{dx} xM(x)&=M(x)+xM'(x)=\frac{(1+x^2)\bigl(\Phi(x)-\frac{1}{2}\bigr)+x\phi(x)}{\phi(x)}\\ 
&\geq0
\end{align*}
and $xf_z(x)$ is also increasing on $[0,z]$. 
Thus, using \eqref{Mill} again
\[0\leq xf_z(x)\leq\lim_{y\to\infty} yf_z(y)=2\Phi(z)-1\,,\]
proving (a).\\
Since $M$ is increasing and $N$ is decreasing, we conclude that for each $z\geq0$ 
\begin{equation}\label{bseq3}
\sup_{x\geq0}\abs{f_z(x)}=f_z(z)=\frac{\bigl(1-F(z)\bigr)F(z)}{p(z)}=\frac{\bigl(1-\Phi(z)\bigr)\bigl(2\Phi(z)-1\bigr)}{\phi(z)}=:g(z)\,.
\end{equation}
To show that $g(z)\leq\frac{1}{2}$ for each $z\geq0$, we define
\begin{equation*}
D_1(x):=\frac{1}{2}\phi(x)-\bigl(1-\Phi(x)\bigr)\bigl(2\Phi(x)-1\bigr),\quad x\geq0
\end{equation*} 
and need to show that $D_1\geq0$. We have $D_1(0)=\frac{1}{2\sqrt{2\pi}}>0$ and $\lim_{x\to\infty}D_1(x)=0$. Thus, it suffices to show that $D_1$ is a decreasing function 
on $[0,\infty)$. We have 
\begin{align}\label{bseq4}
 D_1'(x)&=-\frac{1}{2}x\phi(x)+\phi(x)\bigl(2\Phi(x)-1\bigr)-2\phi(x)\bigl(1-\Phi(x)\bigr)\notag\\
&=\phi(x)\bigl(-\frac{x}{2}+4\Phi(x)-3\bigr)=:\phi(x) D_2(x)\,.
\end{align}
Thus, we want to show that $D_2(x)\leq0$ for $x\geq0$. Note that $D_2(0)=-1$ and $\lim_{x\to\infty}D_2(x)=-\infty$. Furthermore, 
\begin{align}
 D_2'(x)&=-\frac{1}{2}+4\phi(x)\label{bseq5}\\
D_2''(x)&=-4x\phi(x)<0,\quad x>0\,.\label{bseq6}
\end{align}
From \eqref{bseq6} we conclude that $D_2$ is strictly concave. Hence, there is a unique $x_0\in(0,\infty)$ such that $D_2(x_0)=\sup_{x\geq0}D_2(x)$ and 
$D_2'(x_0)=0$. By \eqref{bseq5} we have
\begin{align}\label{bseq7}
D_2'(x_0)=0&\Leftrightarrow\phi(x_0)=\frac{1}{8}\notag\\
&\Leftrightarrow -\frac{1}{2}\bigl(\log2+\log\pi\bigr)-\frac{x_0^2}{2}=-3\log2\notag\\
&\Leftrightarrow x_0^2=5\log2-\log\pi=\log\Bigl(\frac{32}{\pi}\Bigr)\notag\\
&\Leftrightarrow x_0=\sqrt{\log\Bigl(\frac{32}{\pi}\Bigr)}=1.52348\ldots .
\end{align}
Since $D_2(x_0)=-0.01701...<0$, the claim of (a) follows.\\
To prove (c) note that since $f_z$ solves Equation \eqref{steineq} for $h=h_z$
\begin{equation}\label{bseq9}
 f_z'(x)=xf_z(x)+1_{(-\infty,z]}(x)-F(z)\,.
\end{equation}
Hence, for $x\leq z$ by \eqref{Minc}, \eqref{bseq9}, \eqref{Mill} and (a) it follows that 
\begin{align}
 0&<f_z'(x)=2\bigl(1-\Phi(z)\bigr)+xf_z(x)\label{bseq12}\\
&\leq2\bigl(1-\Phi(z)\bigr) +zf_z(z)\label{bseq10}\\
&=2\bigl(1-\Phi(z)\bigr)+\bigl(2\Phi(z)-1\bigr)\frac{z\bigl(1-\Phi(z)\bigr)}{\phi(z)}\notag\\
&\leq2\bigl(1-\Phi(z)\bigr)+\bigl(2\Phi(z)-1\bigr)=1\,.\notag
\end{align}
Similarly, for $x>z$ we have  by \eqref{Ndec}, \eqref{bseq9}, \eqref{Mill} and (a) that
\begin{align}
 0&>f_z'(x)=xf_z(x)-2\Phi(z)+1\notag\\
&\geq zf_z(z)-2\Phi(z)+1 \label{bseq11}\\
&=1-2\Phi(z)+\bigl(2\Phi(z)-1\bigr)\frac{z\bigl(1-\Phi(z)\bigr)}{\phi(z)}\notag\\
&\geq1-2\Phi(z)+\bigl(2\Phi(z)-1\bigr)\frac{z^2}{1+z^2}=-\bigl(2\Phi(z)-1\bigr)\frac{1}{1+z^2}\notag\\
&\geq1-2\Phi(z)>-1\,,\notag
\end{align}
which already prove the first statement of (c). Furthermore, taking $x=0$ and letting $z\downarrow0$ in \eqref{bseq12} shows the optimality of the first bound in (c).
For the second part of (c) note that by \eqref{bseq10} and \eqref{bseq11} for all $x,y\geq0$ we have 
\begin{equation*}
\abs{f_z'(x)-f_z'(y)}\leq 2\bigl(1-\Phi(z)\bigr)+zf_z(z)-\Bigl(zf_z(z)-2\Phi(z)+1\Bigr)=1\,.
\end{equation*}
Optimality of this bound follows by first taking $x=0$ and then letting $z\downarrow 0$ and $y\to\infty$.\\
\end{proof}
 
\normalem
\bibliography{srw}{}
\bibliographystyle{alpha}

\end{document}